\documentclass[smallcondensed]{article}
\usepackage{amsmath,amsfonts,amssymb,amscd,amsthm}
\usepackage{latexsym}
\usepackage[T2A]{fontenc}
\usepackage[cp1251]{inputenc}
\usepackage{graphicx}

\theoremstyle{plain} 
\newtheorem{theorem}{Theorem}[section]
\newtheorem{proposition}[theorem]{Proposition}
\newtheorem{conjecture}[theorem]{Conjecture}
\newtheorem{lemma}[theorem]{Lemma}
\newtheorem{corollary}[theorem]{Corollary}
\theoremstyle{definition} 
\newtheorem{definition}[theorem]{Definition}
\theoremstyle{remark} 
\newtheorem{remark}[theorem]{Remark}

\newcommand{\bbR}{\mathbb{R}}

\newcommand{\bbE}{\mathbb{E}}
\newcommand{\bfp}{\mathbf{p}}
\newcommand{\bfq}{\mathbf{q}}

\newcommand{\Vol}{\mathrm{Vol}}
\newcommand{\Bdy}{\mathrm{Bdry}}

\newcommand{\Conv}{\mathrm{Conv}}

\begin{document}

\title{Strict Kneser-Poulsen conjecture for large radii\thanks{Research supported in part by NSF Grant No. DMS–0209595 (USA).}
}

\author{Igors Gorbovickis\thanks{Department of Mathematics, Cornell University, Ithaca, NY 14853, USA. E-mail: igorsgor@math.cornell.edu}
}

\maketitle

\begin{abstract}
In this paper we prove the Kneser-Poulsen conjecture for the case of large radii. Namely, if a finite number of points in Euclidean space $\bbE^n$ is rearranged so that the distance between each pair of points does not decrease, then there exists a positive number $r_0$ that depends on the rearrangement of the points, such that if we consider $n$-dimensional balls of radius $r>r_0$ with centers at these points, then the volume of the union (intersection) of the balls before the rearrangement is not less (not greater) than the volume of the union (intersection) after the rearrangement. Moreover, the inequality is strict whenever the new point set is not congruent to the original one. Also under the same conditions we prove a similar result about surface volumes instead of volumes.

In order to prove the above mentioned results we
use ideas from tensegrity theory
to strengthen the theorem of Sudakov~\cite{Sudakov}, R. Alexander~\cite{Alexander85} and Capoyleas and Pach~\cite{CapPach}, which says that the mean width of the convex hull of a finite number of points does not decrease after an expansive rearrangement of those points. In this paper we show that the mean width increases strictly, unless the expansive rearrangement was a congruence.

We also show that if the configuration of centers of the balls is fixed and the volume of the intersection of the balls is considered as a function of the radius $r$, then the second highest term in the asymptotic expansion of this function is equal to $-M_nr^{n-1}$, where $M_n$ is the mean width of the convex hall of the centers. This theorem was conjectured by Bal\'{a}zs Csik\'{o}s in 2009.
\end{abstract}

\section{Introduction}
\subsection{Notation}
Let $|\ldots|$ be the Euclidean norm. Let $\bfp= (\bfp_1,\ldots , \bfp_N)$ and $\bfq= (\bfq_1,\ldots , \bfq_N)$ be two configurations of $N$ points, where each $\bfp_i\in\bbE^n$ and each $\bfq_i\in\bbE^n$. If for all $1\le i<j\le N$, $|\bfp_i-\bfp_j|\le |\bfq_i-\bfq_j|$, we say that $\bfq$ is an \textit{expansion} of $\bfp$ and $\bfp$ is a \textit{contraction} of $\bfq$. We denote by $B_n(\bfp_i, r_i)$ the closed $n$-dimensional ball of radius $r_i\ge 0$ in $\bbE^n$ about the point $\bfp_i$, and let $\Vol_n$ represent the $n$-dimensional volume.

\subsection{Kneser-Poulsen conjecture and supporting results}
This work was motivated by the following longstanding conjecture independently stated by Kneser in~\cite{Kneser} and Poulsen in~\cite{Poulsen} for the case when $r_1=\dots =r_N$:

\begin{conjecture}\label{KP1}
If $\bfq= (\bfq_1,\ldots , \bfq_N)$ is an expansion of $\bfp= (\bfp_1,\ldots , \bfp_N)$ in $\bbE^n$, then
\begin{equation*}
\Vol_n\left[\bigcup_{i=1}^N B_n(\bfp_i, r_i)\right] \le \Vol_n\left[\bigcup_{i=1}^N B_n(\bfq_i, r_i)\right].
\end{equation*}
\end{conjecture}

A similar conjecture was proposed in~\cite{KleeWagon} by Klee and Wagon:
\begin{conjecture}\label{KP2}
If $\bfq= (\bfq_1,\ldots , \bfq_N)$ is an expansion of $\bfp= (\bfp_1,\ldots , \bfp_N)$ in $\bbE^n$, then
\begin{equation*}
\Vol_n\left[\bigcap_{i=1}^N B_n(\bfp_i, r_i)\right] \ge \Vol_n\left[\bigcap_{i=1}^N B_n(\bfq_i, r_i)\right].
\end{equation*}
\end{conjecture}

Even though 
for $n\ge 3$ these conjectures still remain a mystery, there are results that provide support for them with the additional assumption that there exists a \textit{continuous expansion} from $\bfp$ to $\bfq$ --- a continuous motion $\bfp(t)=(\bfp_1(t),\dots,\bfp_N(t))$, with $\bfp_i(t)\in\bbE^n$ for all $t\in [0,1]$ and $i=1,\dots, N$ such that $\bfp(0)=\bfp$ and $\bfp(1)=\bfq$, and $|\bfp_i(t)-\bfp_j(t)|$ is non-decreasing for all $1\le i<j\le N$.
Assuming that there exists a continuous expansion from $\bfp$ to $\bfq$, Csik\'os in~\cite{Csikos95},~\cite{Csikos98} and~\cite{Csikos01} proves Conjectures~\ref{KP1}~and~\ref{KP2} and similar conjectures on $n$-dimensional sphere and in $n$-dimensional hyperbolic space.

In~\cite{BezdekConnelly} Bezdek and Connelly prove Conjectures~\ref{KP1}~and~\ref{KP2} for $n=2$ by finding a continuous expansion from $\bfp$ to $\bfq$ in a $4$-dimensional Euclidean space and relating higher dimensional volumes with the $2$-dimensional ones. More precisely, they claim that if there exists a piecewise-analytic continuous expansion from $\bfp$ to $\bfq$ in $(n+2)$-dimensional Euclidean space, then Conjectures~\ref{KP1}~and~\ref{KP2} hold.

\subsection{Main results}
All of the above mentioned results are applicable only to special classes of expansive point rearrangements in $\bbE^n$.
In this paper we prove the following theorem, that holds for any expansion and provides additional support of the Kneser-Poulsen conjecture:
\begin{theorem}\label{main1}
If $\bfq= (\bfq_1,\ldots , \bfq_N)$ is an expansion of $\bfp= (\bfp_1,\ldots , \bfp_N)$ in $\bbE^n$, then there exists $r_0>0$ such that for any $r\ge r_0$
\begin{equation}\label{main_eq1}
\Vol_n\left[\bigcup_{i=1}^N B_n(\bfp_i, r)\right] \le \Vol_n\left[\bigcup_{i=1}^N B_n(\bfq_i, r)\right],
\end{equation}
and
\begin{equation}\label{main_eq2}
\Vol_n\left[\bigcap_{i=1}^N B_n(\bfp_i, r)\right] \ge \Vol_n\left[\bigcap_{i=1}^N B_n(\bfq_i, r)\right],
\end{equation}
and if the point configurations $\bfq$ and $\bfp$ are not congruent, then the inequalities are strict.
\end{theorem}

An essential tool in proving Theorem~\ref{main1} is Theorem~\ref{StrictAlex}, which is a strengthening of the result of Sudakov~\cite{Sudakov} later reproved by Alexander in~\cite{Alexander85} and Capoyleas and Pach in~\cite{CapPach}:

We remind that if $K\subset\bbE^n$ is a compact convex set, then $M_n[K]$, the $n$-dimensional mean width of $K$ (up to multiplication by a dimensional constant), can be defined as an integral over the unit $(n-1)$-sphere $S^{n-1}$:
$$
M_n[K] = \int_{S^{n-1}}\max\{\langle x, u\rangle\colon x\in K\}d\sigma(u),
$$
where $\sigma$ is the Lebesgue measure on $S^{n-1}$ and $\langle\cdot,\cdot\rangle$ denotes the dot product in $\bbR^n$.




\begin{theorem}[Sudakov, Alexander, Capoyleas, Pach]\label{CapPachTh}
Let $\bfq= (\bfq_1,\ldots , \bfq_N)$ be an expansion of $\bfp= (\bfp_1,\ldots , \bfp_N)$ in $\bbE^n$, $n\ge 2$. Then
\begin{equation}\label{CapPach_in}
M_{n} [\Conv\{\bfq_1,\dots,\bfq_N\}]\ge M_{n} [\Conv\{\bfp_1,\dots,\bfp_N\}],
\end{equation}
where $\Conv$ stands for the convex hull of a set in $\bbE^n$. 
\end{theorem}

The following theorem that will be proven in Section~\ref{StrictAS} strengthens this result by saying that the equality in~(\ref{CapPach_in}) is obtained only when the point configurations $\bfp$ and $\bfq$ are congruent.
\begin{theorem}\label{StrictAlex}
Let $\bfq= (\bfq_1,\ldots , \bfq_N)$ be an expansion of $\bfp= (\bfp_1,\ldots , \bfp_N)$ in $\bbE^n$, $n\ge 2$. Then
\begin{equation*}
M_{n} [\Conv\{\bfq_1,\dots,\bfq_N\}]\ge M_{n} [\Conv\{\bfp_1,\dots,\bfp_N\}],
\end{equation*}
and if the point configurations $\bfq$ and $\bfp$ are not congruent, then the inequality is strict.
\end{theorem}

\begin{remark}
We note that Theorem~\ref{StrictAlex} is interesting on its own and requires some additional technics in its proof, compared to the known proofs of Theorem~\ref{CapPachTh}. Moreover, strictness of the inequality in Theorem~\ref{StrictAlex} is crucial for the proof of Theorem~\ref{main1}.
\end{remark}
\begin{remark}
The spherical analog of Theorem~\ref{CapPachTh} was proved by K.~Bezdek and R.~Connelly in~\cite{BezCon_Sphere}. It is conjectured that the spherical version of Theorem~\ref{StrictAlex} should also be true.
\end{remark}

In the same paper~\cite{CapPach} Capoyleas and Pach show that starting from sufficiently large $r$ we have
\begin{equation}\label{eq3}
\Vol_n\left[\bigcup_{i=1}^N B_n(\bfp_i, r)\right] = \delta_n r^n+ M_{n} [\Conv\{\bfp_1,\dots,\bfp_N\}]r^{n-1} + o(r^{n-1}),
\end{equation}
where $\delta_n=\Vol_n[B_n(0,1)]$. This observation together with Theorem~\ref{StrictAlex} proves inequality~(\ref{main_eq1}) of Theorem~\ref{main1}. In order to prove inequality~(\ref{main_eq2}), we establish the following theorem that was conjectured by Bal\'{a}zs Csik\'{o}s during our conversation in 2009:
\begin{theorem}\label{CsikCon}
For any configuration of points $\bfp=(\bfp_1,\dots,\bfp_N)$ in $\bbE^n$, $n\ge 2$, the following asymptotic equality takes place:
\begin{equation}\label{eq4}
\Vol_n\left[\bigcap_{i=1}^N B_n(\bfp_i, r)\right] = \delta_n r^n- M_{n} [\Conv\{\bfp_1,\dots,\bfp_N\}]r^{n-1} + o(r^{n-1}),
\end{equation}
as $r$ tends to infinity.
\end{theorem}
We give a proof of this theorem in Section~\ref{voronoi}.

Another corollary can be obtained from Theorem~\ref{StrictAlex} by applying the $n$-dimensional version of Proposition~1 from~\cite{BezConCsikos}:
\begin{corollary}
If $\bfq= (\bfq_1,\ldots , \bfq_N)$ is an expansion of $\bfp= (\bfp_1,\ldots , \bfp_N)$ in $\bbE^n$, then there exists $r_0>0$ such that for any $r\ge r_0$,
\begin{equation*}
M_n\left[\bigcap_{i=1}^N B_n(\bfp_i, r)\right] \ge M_n\left[\bigcap_{i=1}^N B_n(\bfq_i, r)\right],
\end{equation*}
and if the point configurations $\bfq$ and $\bfp$ are not congruent, then the inequality is strict.
\end{corollary}

\subsection{Kneser-Poulsen type theorems for the volume of the boundary of ball configurations}
In Theorem~\ref{AnalV} we show that $\Vol_n\left[\bigcup_{i=1}^N B_n(\bfp_i, r)\right]$ and $\Vol_n\left[\bigcap_{i=1}^N B_n(\bfp_i, r)\right]$ are holomorphic functions of $r$ in some punctured neighborhood of infinity and have a pole of order $n$ at $r=\infty$. This implies that for sufficiently large values of $r$,~(\ref{eq3}) and~(\ref{eq4}) can be written as Laurent series of $r$ and can be differentiated term by term with respect to $r$. On the other hand,
$$
\frac{d}{dr}\Vol_n\left[\bigcup_{i=1}^N B_n(\bfp_i, r)\right]=\Vol_{n-1}\left[\Bdy\left[\bigcup_{i=1}^N B_n(\bfp_i, r)\right]\right]
$$
and
$$
\frac{d}{dr}\Vol_n\left[\bigcap_{i=1}^N B_n(\bfp_i, r)\right]=\Vol_{n-1}\left[\Bdy\left[\bigcap_{i=1}^N B_n(\bfp_i, r)\right]\right],
$$
where $\Bdy[X]$ denotes the boundary of a set $X\subset\bbE^n$. Thus Theorem~\ref{StrictAlex} implies the following result:
\begin{theorem}\label{main2}
If $\bfq= (\bfq_1,\ldots , \bfq_N)$ is an expansion of $\bfp= (\bfp_1,\ldots , \bfp_N)$ in $\bbE^n$, for $n\ge 2$, then there exists $r_0>0$ such that for any $r\ge r_0$
\begin{equation}\label{eq5}
\Vol_{n-1}\left[\Bdy\left[\bigcup_{i=1}^N B_n(\bfp_i, r)\right]\right] \le \Vol_{n-1}\left[\Bdy\left[\bigcup_{i=1}^N B_n(\bfq_i, r)\right]\right],
\end{equation}
and
\begin{equation}\label{eq6}
\Vol_{n-1}\left[\Bdy\left[\bigcap_{i=1}^N B_n(\bfp_i, r)\right]\right] \ge \Vol_{n-1}\left[\Bdy\left[\bigcap_{i=1}^N B_n(\bfq_i, r)\right]\right],
\end{equation}
and if the point configurations $\bfq$ and $\bfp$ are not congruent, then the inequalities are strict.
\end{theorem}

Inequalities~(\ref{eq5}) and~(\ref{eq6}) were proved in~\cite{BezdekConnelly} for all values of $r$, under the additional condition that there exists an analytic expansion from $\bfp$ to $\bfq$. For $n=2$ the question, whether inequality~(\ref{eq6}) holds for any expansion and for all values of $r$ was asked by R.~Alexander in~\cite{Alexander85} and remains open. In~\cite{BezConCsikos} Bezdek, Connelly and Csik\'os give a positive answer to this question when $N\le 4$. On the other hand, in the case of unions of balls Habicht and Kneser provide an example of an expansion, described in~\cite{KleeWagon}, for which inequality~(\ref{eq5}) does not hold for a particular value of $r$. However, Corollary~\ref{main2} implies that by increasing $r$ we can make inequality~(\ref{eq5}) hold.

\section{Strict inequality for the mean width}\label{StrictAS}
In this section we give a proof of Theorem~\ref{StrictAlex}.
\subsection{The case of continuous expansion}
Consider a smooth motion $\bfp(t)= (\bfp_1(t),\dots, \bfp_N(t))$ of a configuration of $N$ points in $\bbE^n$, where $t\in[0,1]$. Let $d_{ij}(t)= |\bfp_i(t)-\bfp_j(t)|$, and let $d_{ij}'(t)$ be the $t$-derivative of $d_{ij}(t)$.
\begin{definition}
A motion $\bfp(t)$ is called \emph{rigid}, if none of the distances $d_{ij}(t)$ changes during the motion. Otherwise the motion $\bfp(t)$ is \emph{non-rigid}.
\end{definition}
\begin{definition}
We call the motion $\bfp(t)$ \textit{properly expansive} if for any $t\in [0,1]$ all derivatives $d_{ij}'(t)$ are non-negative, and each distance $d_{ij}(t)$ either stays constant during the whole motion, or has a strictly positive derivative $d_{ij}'(t)$ for all internal points $t$ of the interval $[0,1]$.
\end{definition}

\begin{proposition}\label{MainCont}
If $\bfp(t)$ is a smooth non-rigid properly expansive motion of $N$ points in $\bbE^n$ defined for $t\in [0,1]$, then there exists an open subinterval $I\subset [0,1]$, such that the function $M_{n}(t)=M_{n} [\Conv\{\bfp_1(t),\dots,\bfp_N(t)\}]$ is continuously differentiable on $I$ and $M_{n}'(t)>0$ for all $t\in I$.
\end{proposition}

Before we give a proof of Proposition~\ref{MainCont}, we would like to formulate two theorems that will be used in the proof. The first theorem is a simplified version of so called ``Cauchy's arm lemma" (Lemma~5 from~\cite{Connelly82}) and the second is Theorem~8.6 from~\cite{Whiteley84}.

\begin{theorem}\label{Cauchy}
Let $\bfq(t)=(\bfq_1(t),\dots,\bfq_n(t))$ be a continuous motion of some point configuration in $\bbE^d$, such that the distances between any two points are non-decreasing functions of $t$.
Assume for some $t=t_0$ the points $\bfq_1(t_0),\dots,\bfq_n(t_0)$ are the vertices of a $2$-dimensional convex $n$-gon, and the distances between any two points connected by some edge of the $n$-gon are constant throughout the motion.
Then the motion $\bfq(t)$ is rigid.
\end{theorem}

\begin{definition}
We will say that a point $\bfp_0\in \bbE^n$ is in a strictly convex position with respect to a set $S\subset\bbE^n$ if it is not contained in the closed convex hull of $S\setminus\{\bfp_0\}$.
\end{definition}

\begin{theorem}\label{Whiteley}
Let $\bfq(t)=(\bfq_1(t),\dots,\bfq_n(t))$ be a continuous motion of some point configuration in $\bbE^d$, such that the affine span of the point configuration $\bfq(t)$ has the same dimension troughout the motion.
Assume for some $t=t_0$ the following conditions are satisfied:

(i) there exists a convex polytope $P$ such that all of its vertices are from the point set $\bfq(t_0)$ and are in a strictly convex position with respect to $\bfq(t_0)$;

(ii) if some point $\bfq_i(t_0)$ is not a vertex of $P$, then it is an interior point of some edge of $P$;

(iii) for any subset of points $\bfq_{i_1}(t_0), \dots, \bfq_{i_k}(t_0)$ that belong to the same edge or $2$-face of $P$, its motion $\bfq_{i_1}(t), \dots, \bfq_{i_k}(t)$ is rigid.

Then the motion $\bfq(t)$ is rigid.
\end{theorem}

\begin{proof}[Proof of Proposition~\ref{MainCont}]
We define $A_1,\dots A_{2^N-1}$ to be all non-empty subsets of the set $\{1,\dots, N\}$. We define the sequence of nested open intervals $\{I_0\supseteq I_1\supseteq\dots\supseteq I_{2^N-1}\}$ in the following way:

(i) $I_0=(0,1)$.

(ii) The interval $I_{i}$ is a subinterval of $I_{i-1}$ consisting of more than one point.

(iii) For each $t\in I_i$ the dimension of the affine span of the point set $B_i(t)=\{\bfp_j(t) \mid j\in A_{i}\}$ is constant and does not depend on $t$.

Since the maximality of the dimension of the affine span is an open condition, we can always satisfy properties (ii) and (iii) by chosing $I_i$ as a subinterval of the set of such $t\in I_{i-1}$ for which  the affine span of $B_i(t)$ has maximal dimension.

Finally, we chose $I=I_{2^N-1}$. Thus the interval $I$ satisfies the property that for any $i=1,\dots, 2^N$ the dimension of the affine span of the set $B_i(t)$ is constant and does not depend on $t$.

Let $n_0$ be the dimension of the affine span of all points $\bfp_1(t),\bfp_2(t),\dots,\bfp_N(t)$ for $t\in I$. This means that for each $t\in I$ the point configuration $\bfp(t)$ is contained in some $n_0$-dimensional affine subspace $H(t)$ of $\bbE^n$, hence the convex hull $\Conv\{\bfp_1(t),\dots,\bfp_N(t)\}$ is an $n_0$-dimensional convex polytope $P(t)\subset H(t)$ and possesses $n_0$-dimensional mean width. Let $E(t)$ be the set of pairs $(i,j)$, such that the points $\bfp_i(t)$ and $\bfp_j(t)$ are connected by an edge of the polytope $P(t)$. Each edge $e_{ij}(t)$ of the polytope $P(t)$ has a positive curvature $\beta_{ij}(t)$ defined in the following way: consider an $(n_0-1)$-dimensional plane $h(t)$ in $H(t)$ that is orthogonal to the edge $e_{ij}(t)$ and intersects it at some point $\bfq_0(t)$ that lies in the relative interior of $e_{ij}(t)$. Then $h(t)\cap P(t)$ is a convex $(n_0-1)$-polytope. We define $\beta_{ij}(t)$ to be the discrete Gaussian curvature of the polytope $h(t)\cap P(t)$ in $h(t)$ at the vertex $\bfq_0(t)$. According to~\cite{Alexander85},
$$
M_{n_0}[\Conv\{\bfp_1(t),\dots,\bfp_N(t)\}]=M_{n_0}[P(t)]=c_{n_0}\sum_{(i,j)\in E(t)}\beta_{ij}(t)d_{ij}(t),
$$
and
$$
M_n[\Conv\{\bfp_1(t),\dots,\bfp_N(t)\}]=c_{n_0,n}M_{n_0}[\Conv\{\bfp_1(t),\dots,\bfp_N(t)\}],
$$
where $c_{n_0}, c_{n_0,n}>0$ are dimensional constants, and the sum in the first formula is taken over all pairs $i,j$, such that the points $\bfp_i(t)$ and $\bfp_j(t)$ are connected by an edge of the polytope $P(t)$.

It is shown in~\cite{Alexander85} that
$$
\frac{d}{dt}M_{n_0}[P(t)]=c_{n_0}\sum_{(i,j)\in E(t)}\beta_{ij}(t)d_{ij}'(t),
$$
so
\begin{equation}\label{Mnp}
M_{n}'(t)=\frac{d}{dt}M_{n} [\Conv\{\bfp_1(t),\dots,\bfp_N(t)\}]=c_{n_0,n}c_{n_0}\sum_{(i,j)\in E(t)}\beta_{ij}(t)d_{ij}'(t).
\end{equation}
According to~(\ref{Mnp}), $M_{n}'(t)\ge 0$ for all $t\in I$. Assume that $M_{n}'(t_0)=0$ for some $t_0\in I$. This is possible only if $d_{ij}'(t_0)=0$ for all pairs $(i,j)\in E(t_0)$. Since the motion $\bfp(t)$ is properly expansive, this means that the edges of the polytope $P(t_0)$ have the same length for all $t\in I$. Note that the other points cannot get closer throughout their motion, so by Theorem~\ref{Cauchy} the motion of any subconfiguration of points that belong to the same $2$-face of the polytope $P(t_0)$, is rigid, when $t\in I$. Now by applying Theorem~\ref{Whiteley} we get that the motion of the set of all vertices of the polytope $P(t_0)$ is rigid for $t\in I$.

Finally assume the point $\bfp_i(t_0)$ is not a vertex of $P(t_0)$. Then there is a subset of the set of all vertices of the polytope $P(t_0)$, such that the point $\bfp_i(t_0)$ lies in the relative interior of the convex hull of these points. According to our choice of the interval $I$, the point $\bfp_i(t)$ stays in the affine span of these vertices for all $t\in I$. Since the distances between this point and the vertices cannot decrease as $t$ increases, it means that the point $\bfp_i(t)$ does not move with respect to the polytope $P(t_0)$ as $t$ changes in $I$. Thus the motion $\bfp(t)$ is rigid on the interval $I$, and since it is a properly expansive motion, it is rigid on the whole interval $[0,1]$. This brings us to a contradiction to the fact the motion $\bfp(t)$ was chosen to be non-rigid.
\end{proof}

\subsection{The case of a general expansion}
\begin{proof}[Proof of Theorem~\ref{StrictAlex}]
Assume that the point configurations $\bfp$ and $\bfq$ are not congruent. Then let $k$ be a positive integer, and we regard $\bbE^n$ as the subset $\bbE^n=\bbE^n\times\{0\}\subset\bbE^n\times\bbE^{k}=\bbE^{n+k}$.
We chose $k$ to be sufficiently large, so that there exists a smooth non-rigid properly expansive motion from $\bfp$ to $\bfq$ in $\bbE^{n+k}$. Denote this motion by $\bfp(t)$. According to Lemma~$1$ from~\cite{BezdekConnelly}, for sufficiently large $k$ such a motion always exists.

It follows from Theorem~\ref{CapPachTh} that
\begin{equation*} 
M_{n+k} [\Conv\{\bfp_1(t_2),\dots,\bfp_N(t_2)\}]\ge M_{n+k} [\Conv\{\bfp_1(t_1),\dots,\bfp_N(t_1)\}],
\end{equation*}
for every pair $t_1$, $t_2$, such that $0\le t_1\le t_2\le 1$. Also according to Proposition~\ref{MainCont}, there exists a time interval $I\subset [0,1]$, such that
$$
\frac{d}{dt}M_{n+k} [\Conv\{\bfp_1(t),\dots,\bfp_N(t)\}]>0,
$$
for all $t\in I$, hence we obtain a strict inequality
\begin{equation*}
M_{n+k} [\Conv\{\bfq_1,\dots,\bfq_N\}]> M_{n+k} [\Conv\{\bfp_1,\dots,\bfp_N\}]
\end{equation*}
Since the sets $\Conv\{\bfq_1,\dots,\bfq_N\}$ and $\Conv\{\bfp_1,\dots,\bfp_N\}$ are contained in $\bbE^n$, their $n$-dimensional mean widths can be obtained from $(n+k)$-dimensional ones dividing the later by a dimensional constant $c_{n,n+k}$. Thus we have
\begin{equation*}
M_{n} [\Conv\{\bfq_1,\dots,\bfq_N\}]> M_{n} [\Conv\{\bfp_1,\dots,\bfp_N\}].
\end{equation*}
\end{proof}

\section{Properties of the volume functions}\label{polytopes}
In this section we develop some tools that we use in order to prove Theorem~\ref{AnalV} which says that the volume functions $\Vol_n\left[\bigcup_{i=1}^N B_n(\bfp_i, r)\right]$ and $\Vol_n\left[\bigcap_{i=1}^N B_n(\bfp_i, r)\right]$ considered as functions of $r$, are analytic in a punctured neighborhood of $r=\infty$ and have a pole of order $n$ at infinity. We define a so called polytope truncated by a ball and show that its volume is a meromorphic function of the radius in some neighborhood of $r=\infty$ with a possible pole at infinity. Since the union or intersection of balls can be represented as a disjoint union of polytopes truncated by balls (truncated Voronoi regions), this will prove Theorem~\ref{AnalV}.

\begin{definition}
We call a nonempty set $P\subset\bbE^n$ a \textit{(convex) polyhedral set}, if it is an intersection of $\bbE^n$ with finitely many (may be zero) closed halfspaces.
\end{definition}

Let $\bfp_0$ be a point in $\bbE^n$ and let $P\subset \bbE^n$ be a convex polyhedral set.
\begin{definition}
We will call the set $\hat P(r) = P \cap B_n(\bfp_0, r)$ a (convex) polytope truncated by a ball of radius $r$ centered at $\bfp_0$, or just a truncated polytope.
\end{definition}

Now the focus of our interest is how the volume of a truncated polytope depends on the radius, so we define the following function which is the $n$-dimensional volume of a truncated polytope $\hat P(r)$ depending on $r$:
$$
V_{P,\bfp_0, n}(r) = \begin{cases}
0,&\text{if $r<0$;}\\
\Vol_n \left[\hat P(r)\right],&\text{if $r\ge 0$.}
\end{cases}
$$
Now we formulate the main result of this section:
\begin{theorem}\label{AnalTh}
The function $V_{P,\bfp_0, n}(r)$ can be extended to a meromorphic function of some punctured neighborhood of $r=\infty$ with possibly a unique pole at infinity of order not greater than $n$.
\end{theorem}
The proof of Theorem~\ref{AnalTh} follows from the next two lemmas.

Let $F_1, \dots, F_m$ be $(n-1)$-dimensional faces of the polyhedral set $P$. For each face $F_i$ we denote by $h_i$ the distance from the point $\bfp_0$ to the hyperplane containing that face, and we put the sign of the face $\epsilon_i =1$, if the point $\bfp_0$ is contained in the halfspace related to that face, or we put $\epsilon_i =-1$ otherwise.

\begin{lemma}
Let $n\ge 2$. Then the function $V_{P,\bfp_0,n}(r)$ satisfies the following differential equation on the interval $r\in (0,+\infty)$:
\begin{equation}\label{DiffEq}
V_{P,\bfp_0,n}(r) = \frac{1}{n}\sum_{i=1}^m \epsilon_i h_i\Vol_{n-1}[F_i \cap B_n(\bfp_0, r)] +
\frac{r}{n}\frac{d}{dr} V_{P,\bfp_0,n}(r).
\end{equation}
\end{lemma}
\begin{proof}
One can see that the volume of a truncated polytope $\hat P(r)$ is the volume of the cone with the vertex at $\bfp_0$ about the spherical part of the boundary $\Bdy[B_n(\bfp_0, r)]\cap P$ of $\hat P(r)$ plus the sum of the volumes of cones with the common vertex at the same point $\bfp_0$ about the planar faces of $\hat P(r)$ multiplied by the sign of the corresponding face of the polyhedral set $P$ (see an example in Figure~\ref{pic1}):
$$
V_{P,\bfp_0,n}(r) = \frac{1}{n}\sum_{i=1}^m \epsilon_i h_i\Vol_{n-1}[F_i \cap B_n(\bfp_0, r)] +
\frac{r}{n} \Vol_{n-1}[\Bdy[B_n(\bfp_0, r)]\cap P].
$$

Finally we note that $\Vol_{n-1}[\Bdy[B_n(\bfp_0, r)]\cap P] = \frac{d}{dr} V_{P,\bfp_0,n}(r)$, which completes the proof of the lemma.
\end{proof}

\begin{figure}
\caption{$2$-dimensional volume of the truncated polytope $ABCD$ is the sum of the volumes of the cone $AOD$ about the spherical part of the boundary of $ABCD$, plus the volumes of the cones $AOB$ and $COD$, minus the volume of the cone $BOC$.}\label{pic1}
\includegraphics[width=100mm]{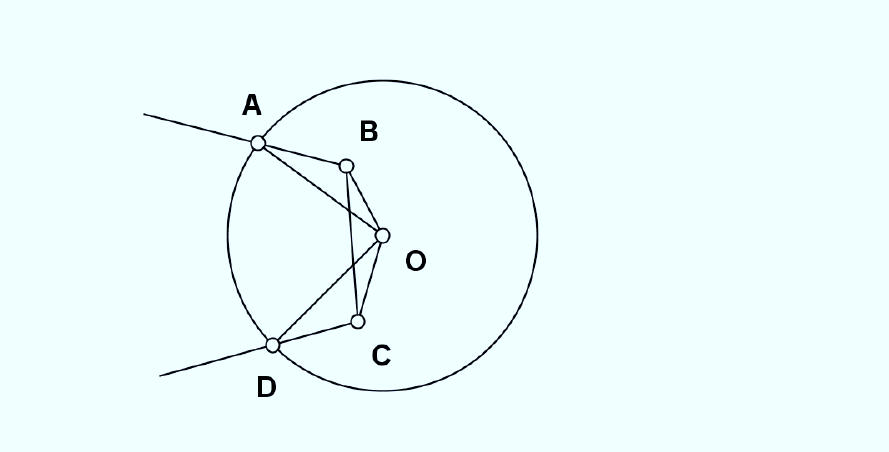}%
\end{figure}

Let the points $\bar\bfp_i$ be the orthogonal projections of the point $\bfp_0$ onto the hyperplanes containing the faces $F_i$. On each of these hyperplanes there is a Euclidian metric induced from the enclosing space, so following the previously chosen notation, one can define the functions
\begin{equation*}
V_{F_i,\bar\bfp_i, n-1}(r) = \begin{cases}
0,&\text{if $r<0$;}\\
\Vol_{n-1} \left[F_i \cap B_{n-1}(\bar\bfp_i, r)\right],&\text{if $r\ge 0$.}
\end{cases}
\end{equation*}
Then for $r\in(\max_{i} h_i,+\infty)$ equation~(\ref{DiffEq}) can be rewritten in the form
\begin{equation*}
V_{P,\bfp_0,n}(r) = \frac{1}{n}\sum_{i=1}^m \epsilon_i h_i V_{F_i,\bar\bfp_i, n-1}\left(\sqrt{r^2-h_i^2}\right) +
\frac{r}{n}\frac{d}{dr} V_{P,\bfp_0,n}(r).
\end{equation*}
If we make a substitution $s=1/r$ and rewrite this equation in terms of the functions
$$
W_{P,\bfp_0,n}(s)=s^nV_{P,\bfp_0,n}(1/s)\quad\text{and}\quad W_{F_i,\bar\bfp_i,n-1}(s)=s^{n-1}V_{F_i,\bar\bfp_i, n-1}(1/s),
$$
we obtain the following equation:
\begin{equation}\label{DiffEqW}
W_{P,\bfp_0,n}'(s)= \sum_{i=1}^m \epsilon_i h_i(1-s^2h_i^2)^{\frac{n-1}{2}}W_{F_i,\bar\bfp_i,n-1}\left(\frac{s}{\sqrt{1-s^2h_i^2}}\right),
\end{equation}
which holds for $s\in(0,\min_i h_i^{-1})$.

\begin{lemma}\label{AnalL}
The function $W_{P,\bfp_0,n}(s)$ can be extended to a holomorphic function in some neighborhood of the point $s=0$.
\end{lemma}
\begin{proof}
We give a proof by induction on the dimension $n$. For $n=1$ and a sufficiently large value of $r$ the function $V_{P,\bfp_0,1}(r)$ is linear for any choice of the point $\bfp_0$ and a polyhedral set $P$, so $W_{P,\bfp_0,1}(s)=sV_{P,\bfp_0,1}(1/s)$ can always be considered as a holomorphic function in some neighborhood of the point $s=0$.

For a general $n>1$, note that the functions $W_{F_i,\bar\bfp_i,n-1}(s)$ are obtained from the volume functions of truncated $(n-1)$-dimensional polytopes in the same way as the function $W_{P,\bfp_0,n}(s)$, so if we assume that the lemma is true in the $(n-1)$-dimensional case, then the functions $W_{F_i,\bar\bfp_i,n-1}\left(\frac{s}{\sqrt{1-s^2h_i^2}}\right)$ are analytic in some neighborhood of $s=0$, and so is the right part of~(\ref{DiffEqW}). This immediately implies that the function $W_{P,\bfp_0,n}(s)$ is also analytic around the point $s=0$.
\end{proof}

\begin{proof}[Proof of Theorem~\ref{AnalTh}]
The proof easily follows from Lemma~\ref{AnalL} and the relation
$$
V_{P,\bfp_0,n}(r)=r^nW_{P,\bfp_0,n}(1/r).
$$
\end{proof}

We will need the following lemma in the proof of Theorem~\ref{CsikCon}
\begin{lemma}\label{WWlemma}
Let $k\le n$ and let $H_1,\dots,H_k\subset\bbE^n$ be closed halfspaces of general position. By $\bar H_1,\dots,\bar H_k$ denote the corresponding complementary closed halfspaces $\bar H_i=\overline{\bbE^n\setminus H_i}$. Let $P=\cap_{i=1}^k H_i$ and $\bar P= \cap_{i=1}^k \bar H_i$. Then for any point $\bfp_0\in\bbE^n$,
$$
W_{P,\bfp_0,n}'(0)+W_{\bar P,\bfp_0,n}'(0)=0.
$$
\end{lemma}
\begin{proof}
Let $F_1, \dots, F_k$ be $(n-1)$-dimensional faces of the polyhedral set $P$ that are contained in the corresponding boundary hyperplanes $\partial H_1,\dots\partial H_k$ of the halfspaces $H_1,\dots,H_k$. Similarly $\bar F_1, \dots, \bar F_k$ are $(n-1)$-dimensional faces of $\bar P$ contained in the same hyperplanes. By $h_i$ we denote the distance from the point $\bfp_0$ to the hyperplane $\partial H_i$, and we put $\epsilon_i =1$, if the point $\bfp_0$ is contained in the halfspace $H_i$, or $\epsilon_i =-1$ otherwise. Let the points $\bar\bfp_i$ be the orthogonal projections of the point $\bfp_0$ onto the hyperplanes $\partial H_i$. Then according to~(\ref{DiffEqW})
$$
W_{P,\bfp_0,n}'(0)+W_{\bar P,\bfp_0,n}'(0)=\sum_{i=1}^m \epsilon_i h_i \left(W_{F_i,\bar\bfp_i,n-1}(0)-W_{\bar F_i,\bar\bfp_i,n-1}(0)\right).
$$
Now in order to prove the lemma, it is sufficient to show that $W_{F_i,\bar\bfp_i,n-1}(0)=W_{\bar F_i,\bar\bfp_i,n-1}(0)$ for all $i$. We can show this in the following way:

Let the point $\bfq_0\subset\bbE^n$ be the orthogonal projection of the point $\bfp_0$ onto the subspace $\cap_{i=1}^k \partial H_i$. If $a_i$ denotes the distance between the points $\bfq_0$ and $\bar\bfp_i$, then the $(n-1)$-dimensional truncated polytope $F_i\cap B_n(\bar\bfp_i,r)$ contains the truncated cone $F_i\cap B_n(\bfq_0,r-a)$ and is contained in the truncated cone $F_i\cap B_n(\bfq_0,r+a)$. Similarly, $\bar F_i\cap B_n(\bar\bfp_i,r)$ contains the truncated cone $\bar F_i\cap B_n(\bfq_0,r-a)$ and is contained in the truncated cone $\bar F_i\cap B_n(\bfq_0,r+a)$. Since the cones $F_i$ and $\bar F_i$ are symmetric (hence congruent), we have the following inequalities:
$$
V_{F_i,\bar\bfq_0, n-1}(r-a_i)\le V_{F_i,\bar\bfp_i, n-1}(r)\le V_{F_i,\bar\bfq_0, n-1}(r+a_i),
$$
$$
V_{F_i,\bar\bfq_0, n-1}(r-a_i)\le V_{\bar F_i,\bar\bfp_i, n-1}(r)\le V_{F_i,\bar\bfq_0, n-1}(r+a_i).
$$
Note that the series expansions of $V_{F_i,\bar\bfq_0, n-1}(r-a_i)$ and $V_{F_i,\bar\bfq_0, n-1}(r+a_i)$ have the same leading term $cr^{n-1}$ (where $c>0$), so from the above inequalities it follows that the series expansions of $V_{F_i,\bar\bfp_i, n-1}(r)$ and $V_{\bar F_i,\bar\bfp_i, n-1}(r)$ also have the same leading term $cr^{n-1}$, so $W_{F_i,\bar\bfp_i,n-1}(0)=W_{\bar F_i,\bar\bfp_i,n-1}(0)=c$.
\end{proof}

\section{Voronoi decomposition}\label{voronoi}

Let $\bfp= (\bfp_1,\ldots , \bfp_N)$ be a configuration of $N$ distinct points in $\bbE^n$.
Consider the $n$-dimensional nearest point and farthest point Voronoi regions for the point configuration $\bfp$:
$$
C_{i,n}=\{\bfp_0\in\bbE^n \mid \text{for all } j, |\bfp_0-\bfp_i| \le |\bfp_0-\bfp_j|\},
$$
$$
C^{i,n}=\{\bfp_0\in\bbE^n \mid \text{for all } j, |\bfp_0-\bfp_i| \ge |\bfp_0-\bfp_j|\}.
$$
Note that the Voronoi regions are convex polyhedral sets. This allows us to define the truncated nearest point and farthest point Voronoi regions $C_{i,n}(\bfp, r) = C_{i,n}\cap B_n(\bfp_i, r)$ and $C^{i,n}(\bfp, r) = C^{i,n}\cap B_n(\bfp_i, r)$ as corresponding truncated polytopes.

In order to simplify the notation, let us introduce the following functions:
$$
V_n(\bfp, r) = \Vol_n\left[\bigcup_{i=1}^N B_n(\bfp_i, r)\right],
$$
$$
V^n(\bfp, r) = \Vol_n\left[\bigcap_{i=1}^N B_n(\bfp_i, r)\right],
$$
$$
W_n(\bfp, s) = s^nV_n(\bfp, 1/s),
$$
$$
W^n(\bfp, s) = s^nV^n(\bfp, 1/s).
$$

\begin{theorem}\label{AnalV}
Consider a point configuration $\bfp=(\bfp_1,\dots,\bfp_N)$ in $\bbE^n$. Then $V_n(\bfp, r)$ and $V^n(\bfp, r)$ are holomorphic functions of $r$ in some punctured neighborhood of infinity and have a pole of order $n$ at $r=\infty$.
\end{theorem}
\begin{proof}
First we notice that
$$
\bigcup_{i=1}^N B_n(\bfp_i, r) = \bigcup_{i=1}^N C_{i,n}(\bfp, r)
$$
and
$$
\bigcap_{i=1}^N B_n(\bfp_i, r) = \bigcup_{i=1}^N C^{i,n}(\bfp, r).
$$
Since intersection of any two distinct truncated Voronoi regions is a set of dimension at most $n-1$, we have
$$
V_n(\bfp, r) = \sum_{i=1}^N\Vol_n\left[C_{i,n}(\bfp, r)\right],
$$
$$
V^n(\bfp, r) = \sum_{i=1}^N\Vol_n\left[C^{i,n}(\bfp, r)\right].
$$
Each truncated Voronoi region is a truncated polytope, so by Theorem~\ref{AnalTh}, their sums $V_n(\bfp, r)$ and $V^n(\bfp, r)$, are analytic in some punctured neighborhood of $r=\infty$. Since these functions grow as $r^n$ when $r$ tends to infinity, they have a pole of order $n$ at $r=\infty$.
\end{proof}

Finally we give a proof of Theorem~\ref{CsikCon}. We start with the following proposition:
\begin{proposition}\label{N1Prop}
Assume $N\le n+1$, and the points $\bfp=(\bfp_1,\dots,\bfp_N)\subset\bbE^n$ are in general position. Then
\begin{equation}\label{WWform}
\left.\frac{d}{ds}\left(W_n(\bfp,s)+W^n(\bfp,s)\right)\right|_{s=0}=0.
\end{equation}
\end{proposition}
\begin{proof}
Each nearest point Voronoi region $C_{i,n}$ is an intersection of some $N-1$ closed halfspaces with boundary hyperpanes orthogonal to the intervals connecting the point $\bfp_i$ with all other points of the point configuration. Then the farthest point Voronoi region $C^{i,n}$ is the intersection of the complementary closed halfspaces. According to Lemma~\ref{WWlemma},
$$
W_{C_{i,n},\bfp_i,n}'(0)+W_{C^{i,n},\bfp_i,n}'(0)=0.
$$
Now the statement of the Proposition follows from the identities
$$
W_n(\bfp,s)=\sum_{i=1}^N W_{C_{i,n},\bfp_i,n}(s),\quad\text{and}\quad W^n(\bfp,s)=\sum_{i=1}^N W_{C^{i,n},\bfp_i,n}(s).
$$
\end{proof}

\begin{proof}[Proof of Theorem~\ref{CsikCon}]
We will prove Theorem~\ref{CsikCon} by showing that the second coefficient in the series for $V^n(\bfp, r)$ is equal to minus the second coefficient in the series for $V_n(\bfp, r)$. In other words, we will show that
\begin{equation}\label{ddelta}
V_n(\bfp, r)+V^n(\bfp, r)=2\delta_nr^n+o(r^{n-1}), \quad\text{as }r\to\infty.
\end{equation}
Then Theorem~\ref{CsikCon} follows from the asymptotic formula~(\ref{eq3}) obtained by Capoyleas and Pach.

Let $k$ be a non-negative integer, and we regard $\bbE^n$ as the subset $\bbE^n=\bbE^n\times\{0\}\subset\bbE^n\times\bbE^{2k}=\bbE^{n+2k}$, where $k$ is chosen so that $N\le n+2k+1$. We consider the point configuration $\bfp$ as a configuration in $\bbE^{n+2k}$, and we show that
\begin{equation}\label{WW2kform}
\left.\frac{d}{ds}\left(W_{n+2k}(\bfp,s)+W^{n+2k}(\bfp,s)\right)\right|_{s=0}=0.
\end{equation}
Indeed, if the points $\bfp=(\bfp_1,\dots,\bfp_N)$ are in a general position, then it follows from Proposition~\ref{N1Prop}. Now if the points $\bfp=(\bfp_1,\dots,\bfp_N)$ are not in a general position, then we can consider an arbitrarily small perturbation of the configuration $\bfp$ in $\bbE^{n+2k}$, such that the perturbed configuration is in a general position. Hence, in order to prove~(\ref{WW2kform}), it is sufficient to show that the left side of~(\ref{WW2kform}) depends continuously on $\bfp\subset\bbE^{n+2k}$. We show this in the following way:

Let $\tilde\bfp=(\tilde\bfp_1,\dots, \tilde\bfp_N)\subset\bbE^{n+2k}$ be a configuration of $N$ points, such that for every index $i=1,\dots, N$, the inequality $|\tilde\bfp_i-\bfp_i|\le a$ holds for some positive $a\in\bbR$. Then the union (intersection) of balls of radius $r$ centered at $\tilde\bfp$, is contained in the union (intersection) of balls of radius $r+a$, centered at $\bfp$. Thus for all sufficiently small $s>0$,
\begin{multline}\label{ineq_cont1}
\left|W_{n+2k}(\tilde\bfp,s)-W_{n+2k}(\bfp,s)\right| \le s^{n+2k}(V_{n+2k}(\bfp,\frac{1}{s}+a)-V_{n+2k}(\bfp,\frac{1}{s}))\le \\
Ca\left(\frac{1}{s}+a\right)^{n+2k-1}s^{n+2k},
\end{multline}
where $C$ is a positive constant, independent from $s$ and $a$. The last inequality in~(\ref{ineq_cont1}) can be obtained, for example, by applying Theorem~\ref{AnalV}. Similarly one can show that
\begin{equation}\label{ineq_cont2}
\left|W^{n+2k}(\tilde\bfp,s)-W^{n+2k}(\bfp,s)\right| \le Ca\left(\frac{1}{s}+a\right)^{n+2k-1}s^{n+2k}
\end{equation}
Now, using inequalities~(\ref{ineq_cont1}) and~(\ref{ineq_cont2}), we get
\begin{multline*}
\left|\left.\frac{d}{ds}\left(W_{n+2k}(\tilde\bfp,s)+W^{n+2k}(\tilde\bfp,s)\right)\right|_{s=0}-\right.\\ \left.\left.\frac{d}{ds}\left(W_{n+2k}(\bfp,s)+W^{n+2k}(\bfp,s)\right)\right|_{s=0}\right|=
\end{multline*}
\begin{multline*}
\left|\lim_{s\to 0}\frac{1}{s}\left(W_{n+2k}(\tilde\bfp,s)+W^{n+2k}(\tilde\bfp,s)-2\delta_{n+2k}\right) -\right.\\
\left.\lim_{s\to 0}\frac{1}{s}\left(W_{n+2k}(\bfp,s)+W^{n+2k}(\bfp,s)-2\delta_{n+2k}\right)\right| =
\end{multline*}
\begin{multline*}
\lim_{s\to 0}\frac{1}{s} \left|(W_{n+2k}(\tilde\bfp,s)-W_{n+2k}(\bfp,s))+(W^{n+2k}(\tilde\bfp,s)-W^{n+2k}(\bfp,s))\right|\le\\
\lim_{s\to 0} 2Ca\left(\frac{1}{s}+a\right)^{n+2k-1}s^{n+2k-1}=2Ca,
\end{multline*}
which proves that the left side of~(\ref{WW2kform}) depends continuously on $\bfp$. Thus~(\ref{WW2kform}) is proved.


It is not difficult to see that~(\ref{WW2kform}) implies
\begin{equation}\label{ddelta2k}
V_{n+2k}(\bfp, r)+V^{n+2k}(\bfp, r)=2\delta_{n+2k}r^{n+2k}+o(r^{n+2k-1}), \quad\text{as }r\to\infty.
\end{equation}
In order to obtain~(\ref{ddelta}) from~(\ref{ddelta2k}), we use the following lemma which is an easy consequence of Lemma~$7$ from~\cite{BezdekConnelly}:
\begin{lemma}\label{l3}
Let $\bfp= (\bfp_1,\ldots , \bfp_N)$ be a fixed configuration of points in $\bbE^n\subset\bbE^{n+2}$. Then
$$
V_n(\bfp, r)=\frac{1}{2\pi r}\frac{d}{dr}V_{n+2}(\bfp, r),
$$
and
$$
V^n(\bfp, r)=\frac{1}{2\pi r}\frac{d}{dr}V^{n+2}(\bfp, r).
$$
\end{lemma}
The configuration $\bfp\subset\bbE^{n+2k}$ is contained in some $n$-dimensional subspace, so according to Lemma~\ref{l3}, we can apply the operator $\frac{1}{2\pi r}\frac{d}{dr}$ to the left and right parts of~(\ref{ddelta2k}) $k$ times. Since according to Theorem~\ref{AnalV} the functions $V_n(\bfp, r)$ and $V^n(\bfp, r)$ are analytic in $r$, the left and the right parts of the identity will remain equal, so this proves~(\ref{ddelta}).
\end{proof}

\section{Acknowledgments}
The author would like to thank R. Connelly, B. Csik\'{o}s and anonymous referees for valuable comments and suggestions.

\bibliography{large_r}

\begin{thebibliography}{10}

\bibitem{Alexander85}
R.~Alexander.
\newblock Lipschitzian mappings and total mean curvature of polyhedral
  surfaces. {I}.
\newblock {\em Trans. Amer. Math. Soc.}, 288(2):661--678, 1985.

\bibitem{BezdekConnelly}
K.~Bezdek and R.~Connelly.
\newblock Pushing disks apart---the {K}neser-{P}oulsen conjecture in the plane.
\newblock {\em J. Reine Angew. Math.}, 553:221--236, 2002.

\bibitem{BezCon_Sphere}
K.~Bezdek and R.~Connelly.
\newblock The {K}neser-{P}oulsen conjecture for spherical polytopes.
\newblock {\em Discrete Comput. Geom.}, 32(1):101--106, 2004.

\bibitem{BezConCsikos}
K.~Bezdek, R.~Connelly, and B.~Csik{\'o}s.
\newblock On the perimeter of the intersection of congruent disks.
\newblock {\em Beitr\"age Algebra Geom.}, 47(1):53--62, 2006.

\bibitem{CapPach}
V.~Capoyleas and J.~Pach.
\newblock On the perimeter of a point set in the plane.
\newblock In {\em Discrete and computational geometry ({N}ew {B}runswick, {NJ},
  1989/1990)}, volume~6 of {\em DIMACS Ser. Discrete Math. Theoret. Comput.
  Sci.}, pages 67--76. Amer. Math. Soc., Providence, RI, 1991.

\bibitem{Connelly82}
R.~Connelly.
\newblock Rigidity and energy.
\newblock {\em Invent. Math.}, 66(1):11--33, 1982.

\bibitem{Csikos95}
B.~Csik{\'o}s.
\newblock On the {H}adwiger-{K}neser-{P}oulsen conjecture.
\newblock In {\em Intuitive geometry ({B}udapest, 1995)}, volume~6 of {\em
  Bolyai Soc. Math. Stud.}, pages 291--299. J\'anos Bolyai Math. Soc.,
  Budapest, 1997.

\bibitem{Csikos98}
B.~Csik{\'o}s.
\newblock On the volume of the union of balls.
\newblock {\em Discrete Comput. Geom.}, 20(4):449--461, 1998.

\bibitem{Csikos01}
B.~Csik{\'o}s.
\newblock On the volume of flowers in space forms.
\newblock {\em Geom. Dedicata}, 86(1-3):59--79, 2001.

\bibitem{KleeWagon}
V.~Klee and S.~Wagon.
\newblock {\em Old and new unsolved problems in plane geometry and number
  theory}, volume~11 of {\em The Dolciani Mathematical Expositions}.
\newblock Mathematical Association of America, Washington, DC, 1991.

\bibitem{Kneser}
M.~Kneser.
\newblock Einige {B}emerkungen \"uber das {M}inkowskische {F}l\"achenmass.
\newblock {\em Arch. Math. (Basel)}, 6:382--390, 1955.

\bibitem{Poulsen}
E.~T. Poulsen.
\newblock Problem 10.
\newblock {\em Math. Scand.}, 2:346, 1954.

\bibitem{Sudakov}
V.~N. Sudakov.
\newblock Gaussian random processes, and measures of solid angles in {H}ilbert
  space.
\newblock {\em Dokl. Akad. Nauk SSSR}, 197:43--45, 1971.

\bibitem{Whiteley84}
W.~Whiteley.
\newblock Infinitesimally rigid polyhedra. {I}. {S}tatics of frameworks.
\newblock {\em Trans. Amer. Math. Soc.}, 285(2):431--465, 1984.

\end{thebibliography}
\bibliographystyle{abbrv} 

\end{document}